\documentclass{amsart}

\usepackage{amsfonts}

\usepackage{amsthm}

\usepackage{amssymb}

\usepackage{amsmath}

\usepackage{enumerate}

\usepackage{verbatim}

\usepackage[all]{xy}

\usepackage{graphicx}

\newcommand{\N}{\mathbb{N}}

\newcommand{\Z}{\mathbb{Z}}

\newcommand{\C}{\mathbb{C}}

\theoremstyle{definition} 

\theoremstyle{remark} 

\theoremstyle{plain} 

\theoremstyle{plain} 

\theoremstyle{remark} 

\theoremstyle{plain} 

\theoremstyle{plain} 

\theoremstyle{plain} 

\theoremstyle{remark} 

	\title[$C^*$-algebras of non-amenable groupoids]{A non-amenable groupoid whose maximal and reduced $C^*$-algebras are the same}
		\author{Rufus Willett}
		\address{Rufus Willett\\
		University of \Hawaii\ at \Manoa, Department of Mathematics, 2565
  McCarthy Mall, Honolulu, HI 96822-2273, USA}
                \email{rufus@math.hawaii.edu}
      		\urladdr{http://math.hawaii.edu/~rufus/}

\newcommand{\Manoa}{M\=anoa}

\newcommand{\Hawaii}{Hawai\kern.05em`\kern.05em\relax i}

\newcommand{\red}{\text{red}}

\theoremstyle{plain}
\newtheorem{theorem}{Theorem}[section]
\newtheorem{lemma}[theorem]{Lemma}

\newtheorem{definition-theorem}[theorem]{Definition / Theorem}

\newtheorem*{conjecture*}{Conjecture}
\newtheorem*{theorem*}{Theorem}

\theoremstyle{definition}
\newtheorem{definition}[theorem]{Definition}

\theoremstyle{remark}
\newtheorem{remark}[theorem]{Remark}
\newtheorem{remarks}[theorem]{Remarks}

\newtheorem*{example*}{Example}  
\newtheorem*{remark*}{Remark}

\begin{document}

\maketitle

\begin{abstract}
We construct a locally compact groupoid with the properties in the title.  Our example is based closely on constructions used by Higson, Lafforgue, and Skandalis in their work on counterexamples to the Baum-Connes conjecture.  It is a bundle of countable groups over the one point compactification of the natural numbers, and is Hausdorff, second countable and \'{e}tale.
\end{abstract}

\section{Introduction}

Say $G$ is a locally compact group equipped with a Haar measure, and $C_c(G)$ is the convolution $*$-algebra of continuous, compactly supported, complex-valued functions on $G$.  In general $C_c(G)$ has many interesting $C^*$-algebra completions, but the two most important are: the maximal completion $C^*_{\max}(G)$, which is the completion taken over (the integrated forms of) all unitary representations of $G$; and the reduced completion $C^*_{\red}(G)$, which is the completion of $C_c(G)$ for the left regular representation on $L^2(G)$.   An important theorem of Hulanicki \cite{Hulanicki:1967lt} says that $C^*_{\max}(G)=C^*_\red(G)$ if and only if $G$ is amenable.

Now, say $G$ is a locally compact, Hausdorff groupoid.  To avoid measure-theoretic complications in the statements below, we will assume that $G$ is \'{e}tale.  Much as in the case of groups, there is a canonically associated convolution $*$-algebra $C_c(G)$ of continuous compactly supported functions on $G$, and this $*$-algebra completes in a natural way to maximal and reduced $C^*$-algebras $C^*_{\max}(G)$ and $C^*_{\text{red}}(G)$.  There is also a natural definition of topological amenability due to Renault \cite[discussion below Definition II.3.6]{Renault:1980fk}  that generalises the definition for groups.   See \cite{Renault:1980fk} and \cite{Anantharaman-Delaroche:2000mw} for comprehensive treatments of groupoids and their $C^*$-algebras, and \cite[Section 2.3]{Renault:2009zr} and \cite[Section 5.6]{Brown:2008qy} for self-contained introductions covering only the (much simpler) \'{e}tale case.

For groupoids, one has the following analogue of one direction of Hulanicki's theorem, for which we refer to book of Anantharaman-Delaroche and Renault \cite[Proposition 3.3.5 and Proposition 6.1.8]{Anantharaman-Delaroche:2000mw}\footnote{Anantharaman-Delaroche and Renault actually prove this result in the much more sophisticated case that $G$ is not \'{e}tale, and thus need additional assumptions about Haar systems.  See \cite[Corollary 5.6.17]{Brown:2008qy} for a self-contained proof of Theorem \ref{gpdhul} as stated here.}.

\begin{theorem}\label{gpdhul}
Say $G$ is a locally compact, Hausdorff, \'{e}tale groupoid.  If $G$ is topologically amenable, then $C^*_{max}(G)=C^*_{\red}(G)$.
\end{theorem}

In this note, we show that the converse to this result is false, so Hulanicki's theorem does not extend to groupoids.

\begin{theorem}\label{main}
There exist locally compact, Hausdorff, second countable, \'{e}tale groupoids with compact unit space that are not topologically amenable, but that satisfy $C^*_{max}(G)=C^*_\red(G)$.
\end{theorem}

\begin{remark}\label{mw amen}
Briefly looking outside the \'{e}tale case, there are (at least)\footnote{Another definition is that of \emph{Borel amenability} from \cite[Definition 2.1]{Renault:2013lp}.  The condition `$C^*_{\max}(G)=C^*_\red(G)$' could also be thought of as a form of amenability, and is sometimes (see e.g.\ \cite{Sims:2013ov}) called \emph{metric amenability}.} 
two notions of amenability that are used for locally compact groupoids: \emph{topological amenability} \cite[Definition 2.2.8]{Anantharaman-Delaroche:2000mw} and \emph{measurewise amenability} \cite[Definition 3.3.1]{Anantharaman-Delaroche:2000mw}.  For a general locally compact, second countable, Hausdorff groupoid $G$ with continuous Haar system, one has the following implications
\begin{align*}
\text{$G$ is topologically amenable} & ~~\Rightarrow~~\text{$G$ is measurewise amenable } \\& ~~ \Rightarrow~~  C^*_{\max}(G)=C^*_\red(G)~;
\end{align*}
for the first implication, see \cite[Proposition 3.3.5]{Anantharaman-Delaroche:2000mw}, and for the second see \cite[Proposition 6.1.8]{Anantharaman-Delaroche:2000mw}.  If $G$ is a group (more generally, a transitive groupoid \cite[Lemma 8]{Sims:2013ov}), all three of these conditions are equivalent.

If $G$ is \'{e}tale, then it is automatically equipped with a continuous Haar system, and topological and measurewise amenability are equivalent \cite[Corollary 3.3.8]{Anantharaman-Delaroche:2000mw}.  Thus our examples show that measurewise amenability does not imply equality of the maximal and reduced $C^*$-algebras of a groupoid.  They do not, however, seem to have any bearing on the general relationship between topological and measurewise amenability.
\end{remark}

The examples used in Theorem \ref{main} are a slight adaptation of counterexamples to the Baum-Connes conjecture for groupoids constructed by Higson, Lafforgue, and Skandalis \cite[Section 2]{Higson:2002la}.  The essential extra ingredient needed is \emph{property FD} of Lubotzky and Shalom \cite{Lubotzky:2004xw}.  Our examples are of a particularly simple form, and are in fact a bundle of groups: see Section \ref{results sec} below for details of all this.  

The existence of examples as in Theorem \ref{main} is a fairly well-known question (see for example \cite[Remark 6.1.9]{Anantharaman-Delaroche:2000mw}, \cite[Section 4.2]{Renault:2009zr}, or \cite[page 162]{Brown:2008qy}) and the answer did not seem to be known to experts.  We thus thought Theorem \ref{main} was worth publicizing despite the similarity to existing constructions.  We should remark, however, that our results seem to have no bearing on the existence of transformation groupoids, or of principal groupoids, with the properties in Theorem \ref{main}.  In particular, we cannot say anything about whether or not equality of maximal and reduced crossed products for a group action implies (topological) amenability of the action.  See Section \ref{comment sec} for some comments along these lines.

This paper is structured as follows. Section \ref{results sec} recalls the construction of Higson, Lafforgue and Skandalis, and proves Theorem \ref{main}.  The proof proceeds by characterizing when the class of examples studied by Higson, Lafforgue and Skandalis are amenable, and when their maximal and reduced groupoid $C^*$-algebras are the same; we then use property FD of Lubotzky and Shalom to show that examples `between' these two properties exist.  Section \ref{comment sec} collects some comments and questions about exactness, transformation groupoids, and coarse groupoids.

\subsection*{Acknowledgements}

The author was partially supported by the US NSF.  The author is grateful to the Erwin Schr\"{o}dinger Institute in Vienna for its support during part of the work on this paper, to Erik Guentner and J\'{a}n \v{S}pakula for interesting conversations on related issues, and to Claire Anantharaman-Delaroche and Dana Williams for very useful comments on an earlier version of this note.

\section{Main result}\label{results sec}

We first recall a construction of a class of groupoids from \cite[Section 2]{Higson:2002la}; some groupoids from this class will have the properties in Theorem \ref{main}.  The starting point for this construction is the following data. 

\begin{definition}\label{approx}
Let $\Gamma$ be a discrete group.  An \emph{approximating sequence} for $\Gamma$ is a sequence $(K_n)$ of subgroups of $\Gamma$ with the following properties.
\begin{enumerate}[(i)]
\item Each $K_n$ is a normal, finite index subgroup of $\Gamma$.
\item The sequence is nested: $K_n\supseteq K_{n+1}$ for all $n$.
\item The intersection of the sequence is trivial: $\bigcap_nK_n=\{e\}$.
\end{enumerate}
An \emph{approximated group} is a pair $(\Gamma,(K_n))$ consisting of a discrete group together with a fixed approximating sequence.
\end{definition}

Here then is the construction of Higson, Lafforgue, and Skandalis that we will use.

\begin{definition}\label{hlsgpd}
Let $(\Gamma,(K_n))$ be an approximated group, and for each $n$, let $\Gamma_n:=\Gamma/K_n$ be the associated quotient group, and $\pi_n:\Gamma\to \Gamma_n$ the quotient map.  For convenience, we also write $\Gamma=\Gamma_\infty$ and $\pi_\infty:\Gamma\to\Gamma$ for the identity map.  As a set, define 
$$
G:=\bigsqcup_{n\in \N\cup\{\infty\}} \{n\}\times \Gamma_n.
$$  
Put the topology on $G$ that is generated by the following open sets.
\begin{enumerate}[(i)]
\item For each $n\in \N$ and $g\in \Gamma_n$, the singleton $\{(n,g)\}$ is open. 
\item For each fixed $g\in \Gamma$ and $N\in \N$, the set $\{(n,\pi_n(g))~|~n\in \N\cup\{\infty\}, n>N\}$ is open.
\end{enumerate}
Finally equip $G$ with the groupoid operations coming from the group structure on the `fibres' over each $n$: precisely, the unit space is identified with the subspace $\{(n,g)\in G~|~g=e\}$; the source and range maps are defined by $r(n,g)=s(n,g)=(n,e)$; and composition and inverse are defined by $(n,g)(n,h):=(h,gh)$ and $(n,g)^{-1}:=(n,g^{-1})$.

It is not difficult to check that $G$ as constructed above is a locally compact, Hausdorff, second countable, \'{e}tale groupoid.  Moreover the unit space $G^{(0)}$ of $G$ naturally identifies with the one-point compactification of $\N$.  We call $G$ the \emph{HLS groupoid} associated to the approximated group $(\Gamma,(K_n))$.
\end{definition}

In the rest of this section, we will characterise precisely when HLS groupoids as above are amenable, and when their maximal and reduced groupoid $C^*$-algebras coincide.  We will then use results of Lubotzky and Shalom to show that examples that are `between' these two properties exist, thus completing the proof of Theorem \ref{main}.

Out first lemma characterises when an HLS groupoid $G$ is amenable.  We recall definitions of amenability that are appropriate for our purposes: for the groupoid definition, compare \cite[Definition 5.6.13]{Brown:2008qy} and \cite[Proposition 2.2.13 (ii)]{Anantharaman-Delaroche:2000mw}; for the group definition, see for example \cite[Definition 2.6.3 and Theorem 2.6.8]{Brown:2008qy}.

\begin{definition}\label{amendef}
Let $G$ be a locally compact, Hausdorff, \'{e}tale groupoid.  $G$ is \emph{amenable} if for any compact subset $K$ of $G$ and $\epsilon>0$ there exists a continuous, compactly supported function $\eta:G\to [0,1]$ such that for all $g\in K$
$$
\Big|\sum_{h\in G:s(h)=r(g)}\eta(h)-1\Big|<\epsilon ~~\text{ and }~~\sum_{h\in G:s(h)=r(g)}|\eta(h)-\eta(hg)|<\epsilon.
$$

Let $\Gamma$ be a discrete group.  $\Gamma$ is \emph{amenable} if for any finite subset $F$ of $G$ and $\delta>0$ there exists a finitely supported function $\xi:\Gamma\to [0,1]$ such that for all $g\in F$
$$
\sum_{h\in \Gamma}|\xi(hg)-\xi(h)|<\delta.
$$
\end{definition}

\begin{lemma}\label{amen}
Let $G$ be the HLS groupoid associated to an approximated group $(\Gamma,(K_n))$.  Then $G$ is (topologically) amenable if and only if $\Gamma$ is amenable.
\end{lemma}

\begin{proof}
This is immediate from \cite[Examples 5.1.3 (1)]{Anantharaman-Delaroche:2000mw} or \cite[Theorem 3.5]{Renault:2013lp}.  For the convenience of the reader, however, we also provide a direct proof.  

Assume $G$ is amenable, and let a finite subset $F$ of $\Gamma$ and $\delta>0$ be given.  Let $K$ be the finite (hence compact) subset $\{\infty\}\times F$ of $G$, and let $\eta:G\to [0,1]$ be as in the definition of amenability for $G$ for the compact subset $K$ and error tolerance $\epsilon<\delta/(1+\delta)$.  Write $M=\sum_{g\in \Gamma} \eta(\infty,g)$ (and note that this is at least $1-\epsilon$) and define $\xi:\Gamma\to [0,1]$ by
$$
\xi(g)=\frac{1}{M}\eta(\infty,g);
$$
it is not difficult to check this works.  

Conversely, assume that $\Gamma$ is amenable, and let a compact subset $K$ of $G$ and $\epsilon>0$ be given.  Let $F$ be a finite subset of $\Gamma$ such that $\{n\}\times \pi_n(F)\supseteq K\cap \{n\}\times \Gamma_n$ for all $n$ (compactness of $K$ implies that such a set exists), and let $\xi:\Gamma\to [0,1]$ be as in the definition of amenability for this $F$ and error tolerance $\delta=\epsilon$.  Define $\eta:G\to[0,1]$ by
$$
\eta(n,g)=\sum_{h\in \pi_n^{-1}(g)}\xi(h);
$$
it is again not difficult to check that this works.
\end{proof}

Our next goal is to characterise when the maximal and reduced groupoid $C^*$-algebras of an HLS groupoid $G$ are equal.  General definitions of the maximal and reduced $C^*$-algebras of an \'{e}tale groupoid can be found in \cite[Definition 2.3.18]{Renault:2009zr} or \cite[pages 159 -- 160]{Brown:2008qy}; the next definition specialises these to HLS groupoids.

\begin{definition}\label{gpdalg}
Let $G$ be an HLS groupoid associated to the approximated group $(\Gamma,(K_n))$.  Let $C_c(G)$ denote the space of continuous, compactly supported, complex-valued functions on $G$ equipped with the convolution product and involution defined by
$$
(f_1f_2)(n,g):=\sum_{h\in \Gamma_n} f_1(n,gh^{-1})f_2(n,h),~~~f^*(n,g):=\overline{f(n,g^{-1})}.
$$
The \emph{maximal $C^*$-algebra of $G$}, denoted $C^*_{\max}(G)$, is the completion of $C_c(G)$ for the norm
$$
\|f\|_{\max}:=\sup\{\|\rho(f)\|~|~\rho:C_c(G)\to \mathcal{B}(H) \text{ a $*$-homomorphism }\}.
$$

For $n\in \N\cup\{\infty\}$, define a $*$-representation $\rho_n:C_c(G)\to \mathcal{B}(l^2(\Gamma_n))$ by the formula
$$
(\rho_n(f)\xi)(g)=\sum_{h\in \Gamma_n}f(n,gh^{-1})\xi(h).
$$
The \emph{reduced $C^*$-algebra of $G$}, denoted $C^*_\red(G)$, is the completion of $C_c(G)$ for the norm 
$$
\|f\|_\red:=\sup\{\|\rho_n(f)\|~|~n\in \N\cup\{\infty\}\}.
$$
\end{definition}

Consider now the quotient $*$-homomorphism 
\begin{equation}\label{pi quot}
\psi:C_c(G)\to \C[\Gamma]
\end{equation}
defined by restriction to the fibre at infinity.  Let $G_\N$ denote the open subgroupoid of $G$ consisting of all pairs $(n,g)$ with $n\in \N$.  The kernel of $\psi$ is equal to the $*$-algebra
$$
C_{c,0}(G_\N):=\{f\in C_c(G)~|~f(\infty,g)=0\text{ for all } g\in \Gamma\}.
$$
Before proving our characterisation of when $C^*_{\max}(G)=C^*_\red(G)$, we will need the following ancillary lemma.

\begin{lemma}\label{unicom}
The $*$-algebra $C_{c,0}(G_\N)$ has a unique $C^*$-algebra completion, which canonically identifies with the $C^*$-algebraic direct sum $\bigoplus_n C^*_\red(\Gamma_n)$ (equivalently, with the reduced groupoid $C^*$-algebra $C^*_\red(G_\N)$).
\end{lemma}

\begin{proof}
Note first that the $*$-subalgebra $C_c(G_\N)$ of $C_{c,0}(G_\N)$ is isomorphic to the algebraic direct sum $\bigoplus_{n\in \N} \C[\Gamma_n]$ of group algebras; in particular, it is isomorphic to an algebraic direct sum of matrix algebras, so has a unique $C^*$-algebra completion: the $C^*$-algebraic direct sum $\bigoplus_n C^*_\red(\Gamma_n)$.  For brevity, write $A$ for $\bigoplus_n C^*_\red(\Gamma_n)$ in the rest of the proof.

The $*$-algebra $C_{c,0}(G_\N)$ contains the commutative $C^*$-algebra $C_0(\N)$ as a $*$-subalgebra, and therefore any $C^*$-norm on $C_{c,0}(G_\N)$ restricts to the usual supremum norm on $C_0(\N)$.  Fix $g\in \Gamma$, and say that $f\in C_{c,0}(G_\N)$ is supported in the subset $\{(n,\pi_n(g))\in G~|~n\in \N\}$.  Then for any $C^*$-norm on $C_{c,0}(G)$,
\begin{equation}\label{g set}
\|f^*f\|=\|n\mapsto |f(n,\pi_n(g))|^2\|_{C_0(\N)}=\sup_{n\in \N}|f(n,\pi_n(g))|^2.
\end{equation}

Let now $f$ be an arbitrary element of $C_{c,0}(G_\N)$.  As $f$ is in $C_c(G)$, there is a finite subset $F$ of $\Gamma$ such that $f$ is supported in 
$$
\{(n,g)\in G~|~g\in \pi_n(F)\}.
$$
For fixed $N\in \N$, write $f_N$ for the restriction of $f$ to $\{(n,g)\in G~|~n\leq N\}$.  For any $C^*$-algebra norm $\|\cdot \|$ on $C_{c,0}(G_\N)$, it follows from line \eqref{g set} that 
$$
\|f-f_N\|\leq |F|\sup_{n>N}\sup_{g\in \Gamma_n}|f(n,g)|
$$
and therefore (as $f$ is in $C_{c,0}(G_\N)$), we have that $\|f-f_N\|\to 0$ as $N\to\infty$.  On the other hand, $f_{N}$ is in $C_c(G_\N)$ and therefore $\|f_{N}\|=\|f_{N}\|_{A}$ by uniqueness of the $C^*$-algebra norm on $C_c(G_\N)$.  In particular $(f_N)_{N=1}^\infty$ is Cauchy in the $A$ norm, and its limit in $A$ identifies naturally with $f$.  It follows that $\|f\|_{A}$ makes sense, and that $\|f\|=\|f\|_{A}$, completing the proof.
\end{proof}

\begin{lemma}\label{maxred}
Let $G$ be the HLS groupoid associated to an approximated group $(\Gamma,(K_n))$.  For each $n\in \N\cup\{\infty\}$, let 
$$
\lambda_n:\C[\Gamma]\to\mathcal{B}(l^2(\Gamma_n))
$$
denote the quasi-regular representation induced by the left multiplication action of $\Gamma$ on $\Gamma_n$.  

Then $C^*_{max}(G)=C^*_\red(G)$ if and only if the maximal norm on $\C[\Gamma]$ equals the norm defined by
\begin{equation}\label{norm}
\|x\|:=\sup_{n\in \N\cup\{\infty\}}\|\lambda_n(x)\|
\end{equation}
\end{lemma}

\begin{proof}
Let $\psi:C_c(G)\to \C[\Gamma]$ be the quotient $*$-homomorphism from line \eqref{pi quot}.  The $C^*_{\max}(G)$ and $C^*_\red(G)$ norms on $C_c(G)$ induce the $C^*$-algebra norms
\begin{align*}
\|x\|_{\infty,\max} & :=\inf\{\|f\|_{\max}~|~f\in \psi^{-1}(x)\} \\ \|x\|_{\infty,\red} & :=\inf\{\|f\|_{\red}~|~f\in \psi^{-1}(x)\}
\end{align*}
on $\C[\Gamma]$.  These norms give rise to completions $C^*_{\infty,\max}(\Gamma)$ and $C^*_{\infty,\red}(\Gamma)$ of $\C[\Gamma]$ respectively; moreover, as $*$-representations of $\C[\Gamma]$ pull back to $*$-representations of $C_c(G)$ via the map $\psi:C_c(G)\to\C[\Gamma]$, it is immediate that $C^*_{\infty,\max}(\Gamma)=C^*_{\max}(\Gamma)$.   

Putting this together with Lemma \ref{unicom}, there is a commutative diagram of $C^*$-algebras
\begin{equation}\label{commutes}
\xymatrix{0 \ar[r] & \bigoplus_{n\in \N} C^*_\red(\Gamma_n) \ar[r] \ar@{=}[d] & C^*_{\max}(G) \ar[r] \ar[d] &  C^*_{\max}(\Gamma) \ar[r] \ar[d] & 0 \\  0 \ar[r] & \bigoplus_{n\in \N} C^*_\red(\Gamma_n) \ar[r] & C^*_{\red}(G) \ar[r] &  C^*_{\infty,\red}(\Gamma) \ar[r] & 0 },
\end{equation}
where the vertical maps are the canonical quotients induced by the identity map on $C_c(G)$ and $\C[\Gamma]$, and where the horizontal lines are both exact\footnote{Exactness of the top line is a special case of \cite[Lemma 2.8]{Muhly:1996hc}; in our case exactness of both lines follows directly from the discussion above.}.  The five lemma then gives that $C^*_{\max}(G)=C^*_\red(G)$ if and only if $C^*_{\max}(\Gamma)=C^*_{\infty,\red}(\Gamma)$; to complete the proof, we must therefore show that for any $x\in \C[\Gamma]$ we have 
$$
\|x\|_{\infty,\red}=\sup_{n\in \N\cup\{\infty\}}\|\lambda_n(x)\|.
$$

Fix an element $x\in \C[\Gamma]$.  For some suitably large $N$ and all $n\geq N$, the quotient map $\pi_n:\Gamma\to \Gamma_n$ is injective on the support of $x$.  Define $f$ to be the element of $C_c(G)$ given by
$$
f(n,g)=\left\{\begin{array}{ll} x(h) & n\geq N \text{ or $n=\infty$, and $g= \pi_n(h)$ for $h\in \text{supp}(x)$} \\ 0 & n<N \end{array}\right.,
$$
which is clearly a lift of $x$ for the quotient map $\psi:C_c(G)\to \C[\Gamma]$.  Now, for each $n\in \N$, write $f_n$ for the restriction of $f$ to $\{(m,g)\in G~|~m\geq n \text{ or } m=\infty\}$ and note by consideration of the kernel $\bigoplus_n C^*_\red(\Gamma_n)$ of the quotient $*$-homomorphism $C^*_\red(G)\to C^*_{\infty,\red}(\Gamma)$ that
$$
\|x\|_{\infty,\red}=\limsup_{n\to\infty}\|f_n\|_{\red}.
$$
However, by definition of the reduced norm (Definition \ref{gpdalg}), 
$$
\|f_n\|_\text{red}=\sup_{m\geq n \text{ or } m=\infty}\|\rho_m(f)\|,
$$ 
and by definition of $f$
$$
\sup_{m\geq n \text{ or } m=\infty}\|\rho_m(f)\|=\sup_{m\geq n \text{ or } m=\infty}\|\lambda_m(x)\|.
$$
We conclude that 
$$
\|x\|_{\infty,\red}=\lim_{n\to\infty}\sup_{m\geq n \text{ or } m=\infty}\|\lambda_m(x)\|.
$$
Finally, note that as the groups $K_n$ are nested, $\|\lambda_n(x)\|\geq \|\lambda_m(x)\|$ whenever $n\geq m$, so
$$
\lim_{n\to\infty}\sup_{m\geq n \text{ or } m=\infty}\|\lambda_m(x)\|=\sup_{n\in \N\cup\{\infty\}}\|\lambda_n(x)\|;
$$
this completes the proof.
\end{proof}

Now, combining Lemmas \ref{amen} and \ref{maxred}, to complete the proof of Theorem \ref{main} it suffices to find an approximated group $(\Gamma,(K_n))$ such that $\Gamma$ is not amenable, but such that the norm on $\C[\Gamma]$ from line \eqref{norm} equals the maximal norm.  We will do this in the next lemma, using property \emph{property FD} of Lubotzky and Shalom \cite{Lubotzky:2004xw} as an essential tool.

\begin{lemma}\label{f2fd}
Say $F_2$ is the free group on two generators.  Set 
$$
K_n=\bigcap\{\text{Ker}(\phi)~|~\phi:F_2\to \Gamma \text{ a group homomorphism, } |\Gamma|\leq n\}.
$$
Then the pair $(F_2,(K_n))$ is an approximated group.  Moreover, if $\lambda_n:\C[F_2]\to \mathcal{B}(l^2(\Gamma_n))$ is as in the statement of Lemma \ref{maxred}, then 
$$
\|x\|_{\max}=\sup_{n\in \N\cup\{\infty\}}\|\lambda_n(x)\|
$$
for all $x\in \C[F_2]$.
\end{lemma}

\begin{proof}
Note first that as $F_2$ is finitely generated and as for each $n\in \N$ there are only finitely many isomorphism types of groups of cardinality at most $n$, the collection 
$$
\{\text{Ker}(\phi)~|~\phi:F_2\to F \text{ a group homomorphism, } |F|\leq n\}
$$
of (finite index, normal) subgroups of $F_2$ is finite.  Hence each $K_n$ is of finite index in $F_2$.  Moreover, the sequence $(K_n)$ is clearly nested, and each $K_n$ is clearly normal.  Finally, note that $F_2$ is residually finite, whence $\bigcap_n K_n$ is trivial.  Hence $(K_n)$ is an approximating sequence for $F_2$.

Now, by \cite[Theorem 2.2]{Lubotzky:2004xw}, the collection of all unitary representations of $F_2$ that factor through a finite quotient is dense in the unitary dual of $F_2$.  Equivalently (compare the discussion in \cite[Section 1]{Lubotzky:2004xw}), if we write $S$ for the collection of all irreducible $*$-representations of $\C[F_2]$ that come from unitary representations of $F_2$ that factor through finite quotients, then 
$$
\|x\|_{\max}=\sup\{\|\rho(x)\|~|~\rho\in S\}.
$$
However, any quotient map of $F_2$ with finite image contains all but finitely many of the $K_n$ in its kernel.  Using the basic fact from representation theory that any irreducible representation of a finite group is (unitarily equivalent to) a subrepresentation of the regular representation, we have that any representation in $S$ is (unitarily equivalent to) a subrepresentation of $\lambda_n$ for all suitably large $n$, and are done.
\end{proof}

As $F_2$ is non-amenable, this completes the proof of Theorem \ref{main}.  We conclude this section with a few comments about the sort of approximated groups Lemma \ref{f2fd} applies to.

\begin{remarks}\label{not all}
\begin{enumerate}
\item For many non-amenable $\Gamma$, there are no approximating sequences $(K_n)$ for $\Gamma$ such that the condition `$\|x\|_{\max}=\sup_{n\in \N\cup\{\infty\}}\|\lambda_n(x)\|$  for all $x\in \C[\Gamma]$' holds.  For example, results of Bekka \cite{Bekka:1999kx} show this for $\Gamma=SL(3,\Z)$ and many other higher rank arithmetic groups.
\item On the other hand, there are certainly other non-amenable groups that we could use to replace $F_2$ in Lemma \ref{f2fd}: for example, any fundamental group of a closed Riemann surface of genus at least two would work.   See \cite[Theorem 2.8]{Lubotzky:2004xw} for this and other examples.
\item Even for $\Gamma=F_2$, the specific choice of subgroups in Lemma \ref{f2fd} is important.  For example, say we realise $F_2$ as a finite index subgroup of $SL(2,\Z)$, and choose $K_n$ to be the intersection of the kernel of the reduction map $SL(2,\Z)\to SL(2,\Z/2^n\Z)$ and $F_2$.  Then $(K_n)$ is an approximating sequence for $F_2$ with respect to which $F_2$ has property ($\tau$) - this follows from Selberg's theorem \cite{Selberg:1965qy} as in \cite[Examples 4.3.3 A]{Lubotzky:1994tw}.  Property ($\tau$) in this case means that if we set $S$ to be the collection of all unitary representations of $F_2$ that factor through some quotient $F_2/K_n$, then the trivial representation is isolated in the closure $\overline{S}$ of $S$ taken with respect to the Fell topology.  However, the trivial representation is not isolated in the full unitary dual of $F_2$ (i.e.\ $F_2$ does not have property (T)).  Putting this together, the pair $(F_2, (K_n))$ will not satisfy the condition  `$\|x\|_{\max}=\sup_{n\in \N\cup\{\infty\}}\|\lambda_n(x)\|$  for all $x\in \C[\Gamma]$'.
\end{enumerate}
\end{remarks}

\section{Concluding remarks}\label{comment sec}

We conclude the paper with some (mainly inconclusive!) comments and questions.

\begin{remark}\footnote{This remark grew out of comments of Claire Anantharaman-Delaroche on a preliminary version of this note.} 
Higson, Lafforgue and Skandalis introduced what we have called HLS groupoids as failure of the sequence
$$
\xymatrix{ 0 \ar[r] & C^*_\text{red}(G_\N) \ar[r] & C^*_{\text{red}}(G) \ar[r] & C^*_{\text{red}}(\Gamma) \ar[r] & 0 }
$$
to be exact on the level of $K$-theory for some such $G$ leads to counterexamples to the Baum-Connes conjecture for groupoids.  Similar failures of exactness seem to appear first in \cite[Remark 4.10]{Renault:1991df} (due to Skandalis).  It is clear from the proof of Lemma \ref{maxred} that failure of this sequence to be exact is crucial for our results to hold.  Relatedly, one also has the following lemma.

\begin{lemma}\label{exact}
Let $G$ be an HLS groupoid associated to a countable discrete group $\Gamma$ (and some choice of approximating sequence for $\Gamma$).  The following are equivalent.
\begin{enumerate}[(i)]
\item $\Gamma$ is amenable.
\item $G$ is amenable.
\item $G$ is exact, i.e.\ for any short exact sequence of $G$-$C^*$-algebras 
$$
\xymatrix{ 0 \ar[r] & I \ar[r] & A \ar[r] & B \ar[r] & 0 }
$$
the corresponding sequence 
$$
\xymatrix{ 0 \ar[r] & I\rtimes_{\text{red}}G \ar[r] & A\rtimes_{\red}G \ar[r] & B\rtimes_{\red}G \ar[r] & 0 }
$$
of reduced crossed product $C^*$-algebras is exact.
\item $C^*_{\text{red}}(G)$ is an exact $C^*$-algebra, i.e.\ for any short exact sequence of $C^*$-algebras 
$$
\xymatrix{ 0 \ar[r] & I \ar[r] & A \ar[r] & B \ar[r] & 0 }
$$
the corresponding sequence 
$$
\xymatrix{ 0 \ar[r] & I\otimes C^*_{\text{red}}(G) \ar[r] & A\otimes C^*_{\text{red}}(G) \ar[r] & B\otimes C^*_{\text{red}}(G) \ar[r] & 0 }
$$
of spatial tensor product $C^*$-algebras is exact.
\end{enumerate}
\end{lemma}

\begin{proof}
If $\Gamma$ is amenable, $G$ is amenable by Lemma \ref{amen}, giving (i) implies (ii).  The implication (ii) implies (iii) follows from \cite[Proposition 3.3.5 and Proposition 6.1.10]{Anantharaman-Delaroche:2000mw}, which show that topological amenability implies that the maximal and reduced crossed products are the same, and \cite[Lemma 6.3.2]{Anantharaman-Delaroche:2000mw}, which shows that the maximal crossed product always preserves exact sequences in this sense.  The implication (iii) implies (iv) follows easily by considering trivial actions.  

To complete the circle of implications, say $\Gamma$ is non-amenable; we will show that $C^*_{\text{red}}(G)$ is not an exact $C^*$-algebra.  Let $C^*_{\text{red},\infty}(\Gamma)$ be as in the proof of Lemma \ref{maxred}; as quotients of exact $C^*$-algebras are exact \cite[Theorem 10.2.4]{Brown:2008qy}, it suffices to prove that $C^*_{\text{red},\infty}(\Gamma)$ is not exact.  Now,  \cite[Proposition 3.7.11]{Brown:2008qy} shows that for a residually finite group $\Gamma$, the maximal group $C^*$-algebra $C^*_{\max}(\Gamma)$ is exact if and only if $\Gamma$ is amenable.  However, it is not difficult to see that the proof of \cite[Proposition 3.7.11]{Brown:2008qy} applies essentially verbatim with $C^*_{\text{red},\infty}(\Gamma)$ replacing $C^*_{\max}(\Gamma)$, so $C^*_{\text{red},\infty}(\Gamma)$ is also not exact if $\Gamma$ is not amenable.
\end{proof}

On the other hand, the following theorem appears in work of Kerr \cite[Theorem 2.7]{Kerr:2011fk} and Matsumura \cite{Matsumura:2012aa}.  

\begin{theorem}\label{km}
Say $\Gamma$ is a discrete group acting by homeomorphisms on a compact space $X$, and assume that $\Gamma$ is an exact group.

Then $C(X)\rtimes_{max} \Gamma=C(X)\rtimes_{\text{red}}\Gamma$ if and only if the action of $\Gamma$ on $X$ is amenable. \qed
\end{theorem}
\noindent In groupoid language, the theorem says that if  $X\rtimes \Gamma$ is an exact transformation groupoid (in the sense of Lemma \ref{exact}, part (iii)) associated to an action of a discrete group on a compact space, then $C^*_{\max}(X\rtimes \Gamma)=C^*_{\text{red}}(X\rtimes\Gamma)$ if and only if $X\rtimes\Gamma$ is amenable.

Given this theorem, our example, and Lemma \ref{exact} above, it is natural to ask the following questions.
\begin{enumerate}[(i)]
\item Can one extend Theorem \ref{km} to general exact (\'{e}tale) groupoids? 
\item Are there examples of (non-amenable) transformation groupoids $X\rtimes \Gamma$ such that $C(X)\rtimes_{\max} \Gamma =C(X)\rtimes_\red \Gamma$ where $\Gamma$ is a non-exact group?
\end{enumerate}
Given the current state of knowledge about non-exact groups, (ii) is probably very difficult; our results in this note can be seen as giving a little evidence that such examples might exist, but we do not know enough to speculate either way.
\end{remark}

\begin{remark}
The main result of this note makes the following problem very natural: give a groupoid theoretic characterisation of when the maximal and reduced groupoid $C^*$-algebras are equal.
\end{remark}

\begin{remark}
A metric space $X$ has \emph{bounded geometry} if for each $r>0$ there is a uniform bound on the cardinality of all $r$-balls, and is \emph{uniformly discrete} if there is an absolute lower bound on the distance between any two points.  Associated to such an $X$ are maximal and reduced \emph{uniform Roe algebras} $C^*_{u,\max}(X)$ and $C^*_{u,\red}(X)$.  In coarse geometry, the natural analogue of the question addressed in this note is: is it possible that $C^*_{u,\max}(X)=C^*_{u,\red}(X)$ for some space $X$ without Yu's property A \cite[Section 2]{Yu:200ve}?  Indeed, Skandalis, Tu, and Yu \cite{Skandalis:2002ng} have associated to such $X$ a coarse groupoid $G(X)$ which is topologically amenable if and only if $X$ has property A \cite[Theorem 5.3]{Skandalis:2002ng}.  Moreover, the maximal and reduced uniform Roe algebras of $X$ identify naturally with the maximal and reduced groupoid $C^*$-algebras of $G(X)$.  

This note grew out of an attempt to understand this question for the specific example looked at in \cite[Example 1.15]{Willett:2010ca}.  We were not, however, able to make any progress on the coarse geometric special case of the general groupoid question.   
\end{remark}

\begin{remark}
A groupoid is \emph{principal} if the only elements with the same source and range are the units.  It is natural to ask if examples with the sort of properties in Theorem \ref{main} are possible for principal groupoids: the coarse groupoids mentioned above are a special case, as are many other interesting examples coming from equivalence relations and free actions of groups.  Again, our results seem to shed no light on this question, and we do not know enough to speculate either way.
\end{remark}

\bibliographystyle{abbrv}

\bibliography{Generalbib}

\end{document}